\documentclass[11pt,reqno]{amsart}
\usepackage{graphicx}
\usepackage{float}
\usepackage[caption = false]{subfig}
\usepackage{enumerate}
\usepackage{mathtools}
\usepackage{breqn}
\usepackage{array, geometry, graphicx}
\usepackage{amsmath,amsfonts,paralist,amssymb,amsthm,mathrsfs}
\usepackage{pdfpages}
\usepackage{amsfonts}
\usepackage{amsthm}
\usepackage{array}
\usepackage{mathtools}
\newtheorem{lem}{Lemma}
\newtheorem{theorem}{Theorem}[section]
\newtheorem{corollary}[theorem]{Corollary}

\theoremstyle{definition}

\theoremstyle{remark}
\newtheorem{remark}[theorem]{Remark}
\numberwithin{equation}{section}
\DeclareMathOperator{\RE}{Re} \DeclareMathOperator{\IM}{Im}\DeclareMathOperator{\asin}{sin^{-1}}\DeclareMathOperator{\atan}{tan^{-1}}\DeclareMathOperator{\Arg}{Arg}

\begin{document}
	\title{On Oblique Domains of Janowski Function}
	\thanks{The second author is supported by The Council of Scientific and Industrial Research(CSIR). Ref.No.:08/133(0030)/2019-EMR-I}
	\author{S. Sivaprasad Kumar}
	\address{Department of Applied Mathematics, Delhi Technological University, Delhi--110042, India}
	\email{spkumar@dce.ac.in}
	\author[Pooja Yadav]{Pooja Yadav}
	\address{Department of Applied Mathematics, Delhi Technological University, Delhi--110042, India}
	\email{poojayv100@gmail.com}

	
	\begin{abstract}
		We investigate certain properties of  tilted (oblique) domains, associated with  the Janowski function $(1+Az)/(1+Bz),$ where $A,B\in\mathbb{C}$ with $A\neq B$ and $|B|\leq 1$. We find several bounds for these oblique domains. Further, we extend the various known argument, radius, and subordination results involving Janowski functions with complex parameters.
	\end{abstract}
	
	\maketitle
\section{Introduction}
\noindent Let $\mathcal{H}(\mathbb{D})$ denote the class of analytic functions defined on the open unit disk $\mathbb{D}=\{z\in\mathbb{C}:|z|<1\}$.   Assume
$\mathcal{H}[a,n]:=\{f\in\mathcal{H}(\mathbb{D}):f(z)=a+a_{n}z^n+a_{n+1}z^{n+1}+\cdots\}$, where $n=1,2,\cdots$ with $a\in\mathbb{C}$ 	and $\mathcal{H}_{1}:=\mathcal{H}[1,1]$. Let
$\mathcal{A}_{n}:=\{f\in\mathcal{H}(\mathbb{D}):f(z)=z+a_{n+1}z^{n+1}+a_{n+2}z^{n+2}+\cdots\}$ and   $\mathcal{A}:=\mathcal{A}_{1}$. A subclass of $\mathcal{A}$
consisting of all univalent functions is denoted by $\mathcal{S}$. We say, $f$ is subordinate to
$g$, written as $f\prec g$, if $f(z) = g(\omega(z))$ where $f,$  $g$ are analytic functions and $\omega(z)$ is a Schwarz function.
Moreover, if $g$ is univalent, then $f\prec g$ if and only if $f(\mathbb{D})\subseteq g(\mathbb{D})$ and $f(0) = g(0)$. For $-\pi/2<\lambda<\pi/2$, Wang in ~\cite{wang} introduced the tilted Carath\'{e}odory class by angle $\lambda$ as: \begin{equation}\label{til}\mathcal{P}_{\lambda}:=\bigg\{p\in\mathcal{H}_{1}: e^{i\lambda} p(z)\prec\frac{1+z}{1-z}\bigg\}.\end{equation}
Here $\mathcal{P}_{0}=\mathcal{P}$, the well known Carath\'{e}odory class.  A Janowski function is a bilinear transformation, which was first investigated by Janowski in ~\cite{janowski}. He introduced the class $\mathcal{P}(A,B),$ where $-1\leq B<A\leq1$, which comprises of the set of all  $p$ in $\mathcal{H}_{1}$ such that \begin{equation*}
p(z)\prec \frac{1+Az}{1+Bz}.
\end{equation*} $\mathcal{S}^{*}(A,B),$ the class of Janowski starlike functions and $\mathcal{C}(A,B),$ the class of Janowski convex functions,  consist  of the set of all  $f\in\mathcal{S}$ satisfying \begin{equation*}\label{s[a,b]}
\frac{z f'(z)}{f(z)}\prec\frac{1+Az}{1+Bz}\qquad \text{and}\qquad 1+\frac{zf''(z)}{f'(z)}\prec\frac{1+Az}{1+Bz},
\end{equation*} respectively. For $0\leq\alpha<1$, $\mathcal{S}^{*}(1-2\alpha,-1)=\mathcal{S}^{*}(\alpha)$, $\mathcal{C}(1-2\alpha,-1)=\mathcal{C}(\alpha),$ the classes of starlike and convex functions of order $\alpha$ respectively
and $\mathcal{S}^{*}(1,2\alpha-1)=\mathcal{RS}^{*}(\alpha),$ the class of  starlike functions of reciprocal order $\alpha$, which can also be stated as:
\begin{equation*}
\mathcal{RS}^{*}(\alpha)=\bigg\{f\in\mathcal{S}:\RE\frac{f(z)}{zf'(z)}>\alpha\bigg\}.
\end{equation*} Clearly, $\mathcal{RS}^{*}(\alpha)\subset\mathcal{RS}^{*}(0)=\mathcal{S}^{*}(0)=\mathcal{S}^{*}$. And for $\beta>1,$ $\mathcal{S}^{*}(1-2\beta,-1)=\mathbb{M}(\beta)$, the class of Uralegaddi functions, which is stated as:
\begin{equation*}
\mathbb{M}(\beta)=\left\{f\in\mathcal{S}:\RE\frac{zf'(z)}{f(z)}<\beta\right\}.
\end{equation*}
For $-\pi/2<\alpha<\pi/2$ and $0\leq\beta<1$, $\mathcal{S}^{*}(e^{-i \alpha}\left(e^{-i\alpha}-2\beta\cos\alpha\right),-1)=\mathcal{S}^{*}_{\alpha}(\beta)$, the class of $\alpha-$spirallike of order $\beta.$ From the above classes the range of $A$ in $\mathcal{P}_{\lambda}$, $\mathbb{M}(\beta)$ and $\mathcal{S}^{*}_{\alpha}(\beta)$  is not as given by Janowski. This motivates us to extend and study Janowski function with complex parameters. In the past many authors have studied $(1+Az)/(1+Bz),$ where $A,B\in\mathbb{C}$ with $|B|\leq1$ and $A\neq B.$ For instance see \cite{ali,kuroki,kuroki1}.
In the present paper we investigate the class $\mathcal{P}(A,B,\alpha),$ where $A,B\in\mathbb{C}$ with $|B|\leq1,$ $A\neq B$ and $0<\alpha\leq1$, which includes the set of all $p\in\mathcal{H}_{1}$ satisfying
\begin{equation*}
p(z)\prec\bigg(\frac{1+Az}{1+Bz}\bigg)^{\alpha}.
\end{equation*}
The class of Janowski strongly starlike functions of order $\alpha$, denoted by $\mathcal{SS}^{*}(A,B,\alpha)$, which is the set of all  $f\in\mathcal{S}$ such that \begin{equation*}
\frac{z f'(z)}{f(z)}\prec\bigg(\frac{1+Az}{1+Bz}\bigg)^{\alpha}.
\end{equation*}
In particular,  $\mathcal{SS}^{*}(1,-1,\alpha)=\mathcal{SS}^{*}(\alpha),$ the class of strongly starlike functions of
order $\alpha$, which can also be stated as:
\begin{equation*}
\mathcal{SS}^{*}(\alpha)=\biggl\{f\in\mathcal{S}:\bigg|\arg\frac{zf'(z)}{f(z)}\bigg|<\frac{\alpha\pi}{2}\biggr\}
\end{equation*}
and note that $\mathcal{SS}^{*}(1,-1,1)=\mathcal{SS}^{*}(1)=\mathcal{S}^{*}$. Apart from  the above classes we also establish various subordination, radius, argument problems involving Janowski function with complex parameters. Also we point out many known results in this direction which are especial cases of our result.

\section{Basic properties of the class $\mathcal{P}(A,B,\alpha)$}

In the present section we discuss various geometric properties concerned with functions belonging to $\mathcal{P}(A,B,\alpha).$ We begin with  the following bound estimate result.
\begin{theorem}\label{argjan}
	Let $h\in\mathcal{H}_{1}$ and $A,B\in\mathbb{C}$ with $A\neq B$ and $|B|\leq1$. Further if $|A-B|\leq|1-A\overline{B}|$ and
	\begin{equation}\label{sub}
	h(z)\prec(1-\gamma)\bigg(\frac{1+Az}{1+Bz}\bigg)^{\alpha}+\gamma,
	\end{equation}
	for some $0<\alpha\leq1$, $\gamma\in\mathbb{C}\backslash \{1\}$, then for $|z|=r<1,$ we have
	
	\begin{itemize}	\item[\emph{(i)}] $\bigg|\arg (h(z)-\gamma)-\alpha\atan{\dfrac{\IM(A\overline{B})r^2}{\RE(A\overline{B})r^2-1}}\bigg|<\alpha\asin\dfrac{|A-B|r^2}{|1-A\overline{B}|r^2}$
		\item[\emph{(ii)}]	 $\bigg(\dfrac{|1-A\overline{B}r^2|-|A-B|r}{1-|B|^2r^2}\bigg)^{\alpha}\leq\left|\dfrac{h(z)-\gamma}{1-\gamma}\right|\leq\bigg(\dfrac{|1-A\overline{B}r^2|+|A-B|r}{1-|B|^2r^2}\bigg)^{\alpha}$
		\item[\emph{(iii)}] $M(t_{1}+\pi)\cos(N(t_{1}+\pi))\leq\RE\left(\dfrac{h(z)-\gamma}{1-\gamma}\right)\leq M(t_{1})\cos(N(t_{1}))$
		\item[\emph{(iv)}] $M(\tau-t_{2})\sin(N(\tau-t_{2}))\leq\IM\left(\dfrac{h(z)-\gamma}{1-\gamma}\right)\leq M(t_{2})\sin(N(t_{2})),$
	\end{itemize}	where
	$M(t)=\left(\sqrt{(u(t))^2+(v(t))^2}\right)^\alpha$ and $N(t)=\alpha\atan{\bigg(\dfrac{v(t)}{u(t)}\bigg)},$ with $u(t)=\dfrac{1-\RE(A\overline{B})r^2+|A-B|r\cos{t}}{1-|B|^2r^2}$ and $v(t)=\dfrac{|A-B|r\sin{t}-\IM(A\overline{B})r^2}{1-|B|^2r^2}$.
	Further
	$t_{1}$ and $t_{2}$ are the  roots of \begin{equation}\label{ro}
	\dfrac{u(t)u'(t)+v(t)v'(t)}{u(t)v'(t)-v(t)u'(t)}=\tan\left(\alpha\atan\dfrac{v(t)}{u(t)}\right)\quad \end{equation}
	and
	\begin{equation}\label{roo}
	\dfrac{v(t)u'(t)-u(t)v'(t)}{u(t)u'(t)+v(t)v'(t)}=\tan\left(\alpha\atan\dfrac{v(t)}{u(t)}\right)\quad
	\end{equation} respectively.
\end{theorem}
\begin{proof}
	Let  $H(z):=((h(z)-\gamma)/(1-\gamma))^{\frac{1}{\alpha}}.$ Since $h(z)\neq0$ and $h(0)=1,$ we have $H\in\mathcal{H}_{1}$ and \begin{equation}\label{sub1/a}
	H(z)\prec\frac{1+Az}{1+Bz}.
	\end{equation}
	As $A,B\in\mathbb{C}$, clearly $H(z)$ is contained in the oblique circle, shown in the Figure \ref{fig1},  whose radius and center are given by $R:=|A-B|r/(1-|B|^2r^2)$ and $C:=1-A\overline{B}r^2/(1-|B|^2r^2),$
	respectively with angles $\tau(r):=\atan{(\IM(A\overline{B})r^2/(\RE(A\overline{B})r^2-1))}$ and $\zeta(r):=\asin{(|A-B|r^2/(|1-A\overline{B}|r^2))}.$ By taking  argument estimate   of (\ref{sub1/a}) and using the Figure~\ref{fig1} with the fact  that circle is symmetric about the line passing through origin and center, we obtain $(i)$.  Let $|z|=r<1$ and $t\in[0,2\pi)$,  we have $(h(re^{i t})-\gamma)/(1-\gamma)\in\Omega:=((1+Az)/(1+Bz))^\alpha,$ which implies $\partial{\Omega}:(C+Re^{i t})^\alpha:=M(t)e^{iN(t)}.$
	To find modulus, real and imaginary parts  estimate, we need the critical points. By a simple computation, we obtain $M'(t)=0$ at $t=\tau(1)=\tau$ and $\tau+\pi$, $(M(t)\cos{N(t)})'=0$ at the roots $t_{1}$ and $t_{1}+\pi$  of the equation (\ref{ro}) and $(M(t)\sin{N(t)})'=0$ at the roots $t_{2}$ and $\tau-t_{2}$ of the equation (\ref{roo}), all these values eventually yield $(ii)$, $(iii)$ and $(iv)$ respectively.
\end{proof}
\begin{figure}[h]
	\centering
	\includegraphics[width=110mm, height=75mm]{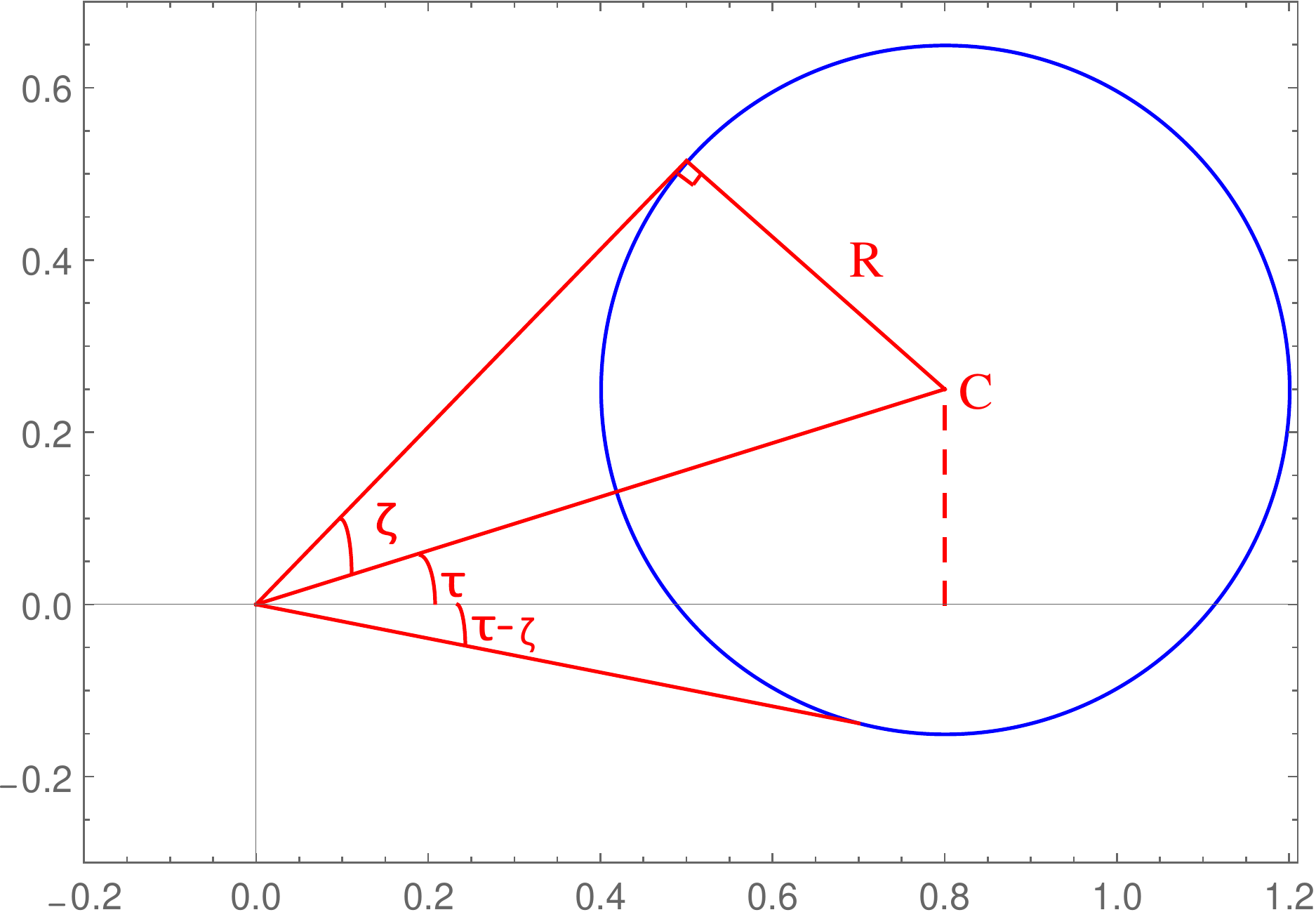}
	\caption{The image of $\partial\mathbb{D}$ under $(1+Az)/(1+Bz)$.}\label{fig1}
\end{figure}
\noindent

\begin{remark}1. When $\gamma=0,$ we can obtain from Theorem \ref{argjan} various bound estimates for functions in  $\mathcal{P}(A,B,\alpha).$\\
	2. By taking $\alpha=1/2$, $\gamma=0$, $A=1$ and $B=0$ in  Theorem \ref{argjan}, we have $h(z)\prec\sqrt{1+z}.$ Therefore the estimates  are $|\arg h(z)|\leq\asin{r}/2$ and $\sqrt{1-r}\leq|h(z)|\leq\sqrt{1+r}.$ If $r=1,$ we have $t_{1}=0$ and  $t_{2}=2\pi/3$,  which implies $0<\RE h(z)<\sqrt{2}$ and $-0.5<\IM h(z)<0.5$.\\
	3. Note that if $w(z)=(1+Az)/(1+Bz)$, then $w(0)=1\in w(\mathbb{D})$. Thus the image domain $w(\mathbb{D})$ will always intersect real axis even if it is an oblique domain, non-symmetric with respect to real axis.\end{remark}
For $\alpha=1$ and $\gamma=0$, Theorem \ref{argjan}, reduces to the following sharp bounds:
\begin{corollary}
	Let $h\in\mathcal{H}_{1}$ and $A,B\in\mathbb{C}$ with $A\neq B$, $|A|\leq1$ and $|B|\leq1$. Further if
	\begin{equation*}
	h(z)\prec\frac{1+Az}{1+Bz},
	\end{equation*}
	then for $|z|=r<1,$ we have
	
	\begin{itemize}	\item[\emph{(i)}] $\bigg|\arg h(z)-\atan{\dfrac{\IM(A\overline{B})r^2}{\RE(A\overline{B})r^2-1}}\bigg|<\asin\dfrac{|A-B|r}{|1-A\overline{B}r^2|}$
		\item[\emph{(ii)}]	 $\dfrac{|1-A\overline{B}r^2|-|A-B|r}{1-|B|^2r^2}\leq|h(z)|\leq\dfrac{|1-A\overline{B}r^2|+|A-B|r}{1-|B|^2r^2}$
		\item[\emph{(iii)}] $\dfrac{1-\RE(A\overline{B})r^2-|A-B|r}{1-|B|^2r^2}\leq\RE h(z)\leq\dfrac{1-\RE(A\overline{B})r^2+|A-B|r}{1-|B|^2r^2}$
		\item[\emph{(iv)}] $\dfrac{\IM(A\overline{B})r^2+|A-B|r}{|B|^2r^2-1}\leq\IM h(z)\leq\dfrac{\IM(A\overline{B})r^2-|A-B|r}{|B|^2r^2-1}$.
	\end{itemize}
\end{corollary}
Note that if $A=a e^{i m \pi}$ and $B=b e^{in \pi},$ where $a\geq0,0\leq b\leq1,-1\leq m,n\leq1$ and $A\neq B$ then $(1+Az)/(1+Bz)$ maps the unit disk onto
\begin{equation}\label{disk}
H(\mathbb{D}):=\bigg\{w\in\mathbb{C}: \bigg|w-\frac{1-A\overline{B}}{1-|B|^2}\bigg|<\frac{|A-B|}{1-|B|^2} \bigg\} .
\end{equation}
Clearly, $H(\mathbb{D})$ represents a disk when $b< 1$ and a half plane when $b=1$. We see that
$w=0$ is an exterior or  interior  or   boundary point of $H(\mathbb{D})$ is decided by the value of $a$ as  $a<1$ or $a>1$ or $a=1$ respectively. Therefore to investigate argument related problems of a Janowski function, we take $0\leq a\leq1$ or  $|A-B|\leq|1-A\overline{B}|$ so that $w=0$ is not an interior point of $H(\mathbb{D})$. Following are the radius and center  of the disk (\ref{disk}):
\begin{gather*}
R:=\frac{|A-B|}{1-|B|^2}=\frac{\sqrt{a^2+b^2-2ab\cos((n-m)\pi)}}{1-b^2},\\
C:=\frac{1-A\overline{B}}{1-|B|^2}=\frac{1-ab\cos((n-m)\pi)+i ab\sin((n-m)\pi)}{1-b^2}.
\end{gather*}
We observe that whenever the difference  of $m$ and $n$ is same  then  the corresponding  $R$ and $C$ also remain same. Therefore  without lose of generality, we can fix $n=1$. Accordingly we confine our study of Janowski function by considering
\begin{equation}\label{takeb}
\frac{1+Az}{1-bz},
\end{equation}
where $ A+b\neq0$, $A\in\mathbb{C},$ $0\leq b\leq1$ and the class $\tilde{P}(A,b)=\{p\in\mathcal{H}_{1}:p(z)\prec(1+Az)/(1-bz)\}.$ The corresponding class of Janowski strongly starlike functions of order $\alpha$ is denoted by $\widetilde{SS}^{*}(A,b,\alpha)=\{f\in\mathcal{S}:zf'(z)/f(z)\prec((1+Az)/(1-bz))^\alpha\}$.
As a consequence of Theorem \ref{argjan}, the following result yields an equivalence relation between half plane Janowski sector whose boundary passes through origin and its argument bounds.

\begin{theorem}\label{lemma2}
	For $0<\alpha\leq1$ and $-1< m<1$, then the function
	\begin{equation}\label{h(z)}
	h(z)=\bigg(\frac{1+e^{i m\pi}z}{1-z}\bigg)^{\alpha},
	\end{equation}
	is analytic, univalent and  convex in $\mathbb{D}$ with
	\begin{equation}\label{rotation}
	h(\mathbb{D})=\bigg\{w\in\mathbb{C}:
	-\alpha(1-m)\frac{\pi}{2}\leq\arg{w}\leq\alpha(1+m)\frac{\pi}{2}\bigg\}.
	\end{equation}
\end{theorem}

\begin{proof}
	By using Theorem \ref{argjan}, the function $h(z)$ given in (\ref{h(z)}) satisfy
	\begin{equation*}
	\bigg|\arg{h(z)}-\alpha\atan\bigg(\tan\frac{m\pi}{2}\bigg)\bigg|<\frac{\alpha\pi}{2},
	\end{equation*}
	which gives the desired result.
\end{proof}

\begin{remark}
	1. From the domain (\ref{rotation}) we have: \begin{equation}\label{remark}(h(\mathbb{D}))^{1/\alpha}=\{w\in\mathbb{C}:\RE e^{-im\pi/2}w>0\}.\end{equation}
	2. For $0<\alpha_{1}\leq1$ and $0<\alpha_{2}\leq1$, if $m=(\alpha_{1}-\alpha_{2})/(\alpha_{1}+\alpha_{2})$ and $\alpha=(\alpha_{1}+\alpha_{2})/2$ in (\ref{h(z)}) then Theorem \ref{lemma2} reduces to ~\cite[Lemma 3]{jin}.
\end{remark}
\noindent The following result generalizes Theorem \ref{lemma2}.
\begin{theorem}\label{lemma3}
	Let  $h\in\mathcal{H}_{1}$  and $0\leq b\leq1$ with $b+e^{im\pi}
	\neq0$, where $-1\leq m\leq1$. Also, if
	\begin{equation}\label{eq1}
	h(z)\prec \frac{1+e^{im\pi} z}{1-b z},
	\end{equation}
	then \begin{equation}\label{eq2}
	\RE	e^{-i\lambda}h(z)>0,
	\end{equation} where $\lambda=\atan\bigg(\dfrac{b\sin{(m\pi)}}{b\cos{(m\pi)}+1}\bigg)$.
\end{theorem}
\begin{proof}
	To obtain (\ref{eq2}),	 it  suffices to show that $|\arg(e^{-i\lambda}w)|<\pi/2$ or $|\arg w-\lambda|<\pi/2$, where $w=h(z)$. By using Theorem \ref{argjan} with $\alpha=1$, the function $h(z)$ given in (\ref{eq1}) satisfy
	\begin{equation*}
	\bigg|\arg h(z)-\atan\frac{b\sin(m\pi)}{b\cos(m\pi)+1}\bigg|<\asin\frac{|e^{im\pi}+b|}{|1+be^{im\pi}|}=\frac{\pi}{2},
	\end{equation*}
	which leads to the desired result.
\end{proof}

\begin{remark}\label{appli}
	1. When $b=1$ then Theorem \ref{lemma3} provides sufficient condition for  functions to be in the class $\mathcal{P}_{-\lambda}.$\\
	2. Let $|A|\leq1$ and $|B|\leq1$ with $A\neq B$.  Assume $a(\alpha):=\arg((1+Az)/(1+Bz))^\alpha.$ Then from Theorem \ref{argjan}, we observe that $\max a(\alpha_{1})\leq\max a(\alpha_{2})$ and $\min a(\alpha_{1})\geq\min a(\alpha_{2})$, whenever $0<\alpha_{1}\leq \alpha_{2}\leq1.$ Therefore we have \begin{equation*}
	\bigg(\frac{1+A z}{1+B z}\bigg)^{\alpha_{1}}\prec\bigg(\frac{1+A z}{1+B z}\bigg)^{\alpha_{2}}\quad (0<\alpha_{1}\leq \alpha_{2}\leq1).
	\end{equation*}
\end{remark}

\begin{theorem}\label{dist}
	Let $|A_{j}|\leq1$ and $|B_{j}|\leq1$ $(j=1,2)$ with $A_{1}\neq B_{1}$ and $A_{2}\neq B_{2}$. Let $C_{j}$ and $R_{j}$ be the centres and radii of $(1+A_{j}z)/(1+B_{j}z)=\phi_{j}(z)$ $(j=1,2)$, respectively.  Then (i) $\mathcal{P}(A_{1},B_{1})\subseteq \mathcal{P}(A_{2},B_{2})$ if and only if the line segment $C_{1}C_{2}$ lies entirely in the domain $\phi_{1}(\mathbb{D})$, whenever $
	|C_{1}-C_{2}|\leq R_{1}$. \\
	(ii) $\mathcal{P}(A_{1},B_{1})\subseteq \mathcal{P}(A_{2},B_{2})$ if and only if the line segment $C_{1}C_{2}$ does not lie entirely in the  domain $\phi_{1}(\mathbb{D})$, whenever $|C_{1}-C_{2}|\geq R_{1}$.
	
\end{theorem}
The proof is skipped here, as the result is evident from the Figure \ref{fig2}.

\begin{figure}[H]
	\centering
	\includegraphics[width=120mm, height=55mm]{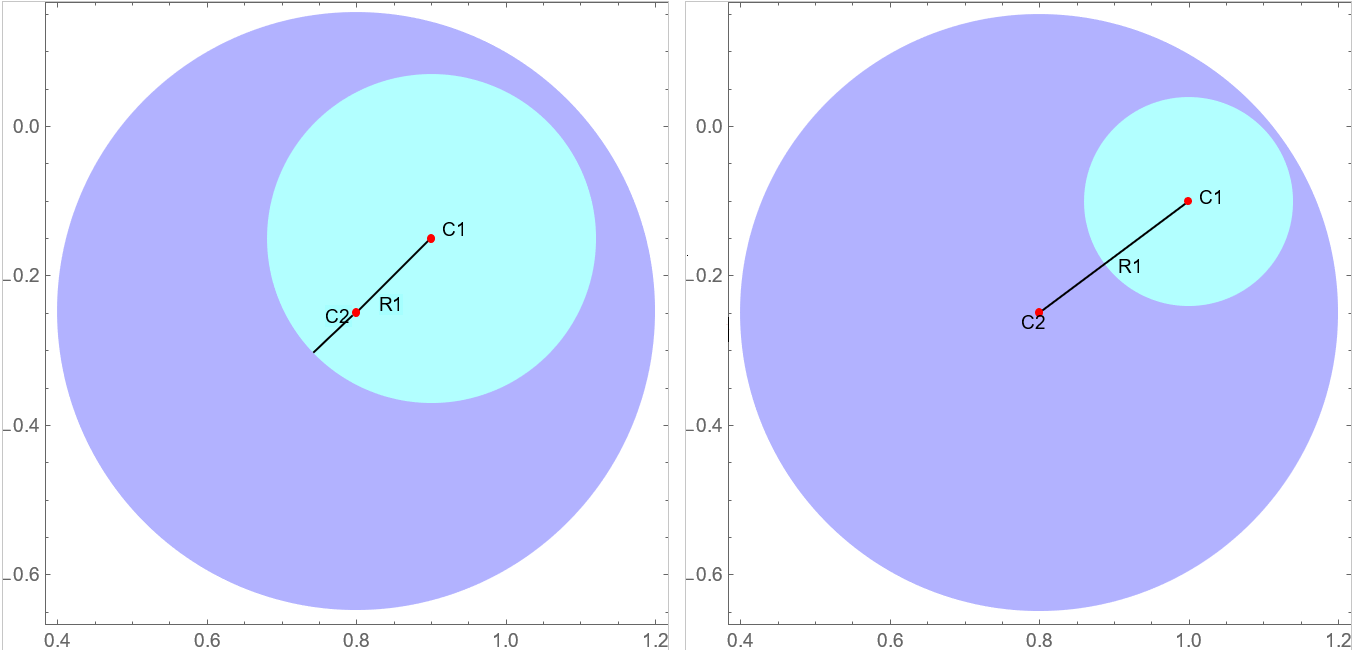}
	\caption{The image of $\mathbb{D}$ under $\phi_{i}(z)$ $(i=1,2)$.}\label{fig2}
\end{figure}
We note that if $\mathcal{P}(A_{1},B_{1})\subseteq \mathcal{P}(A_{2},B_{2}),$ then $(1+A_{1}z)/(1+B_{1}z)\prec(1+A_{2}z)/(1+B_{2}z),$ thus  from Theorem \ref{dist}, we obtain the following result:
\begin{corollary}\label{dist cor}
	If $\mathcal{P}(A_{1},B_{1})\subseteq \mathcal{P}(A_{2},B_{2})$, then for $0<\alpha\leq1$, we also have
	\begin{equation*}
	\bigg(\frac{1+A_{1}z}{1+B_{1}z}\bigg)^\alpha\prec\bigg(\frac{1+A_{2}z}{1+B_{2}z}\bigg)^\alpha.
	\end{equation*}
\end{corollary}

\noindent Note that the observations made in \cite{poonam} are  generalized in part 2 of Remark \ref{appli}  and Corollary \ref{dist cor}.

\section{Argument related results}

In this section, we  generalized various subordination results using  Theorem \ref{lemma2} and the following versions of the  Jack's Lemma:

\begin{lem}\emph{\cite[p. 234-235]{Nunokawal1}}\label{Nunokawa1}
	Let $h\in\mathcal{H}[1,n]$. If there exists a point $z_{0}\in\mathbb{D}$, such that
	\begin{gather*}
	|\arg{p(z)}|<|\arg{p(z_{0})}|=\frac{\pi\beta}{2},\quad (|z|<|z_{0}|) ,\intertext{for some $\beta>0$, then we have}
	\frac{z_{0}p'(z_{0})}{p(z_{0})}=\frac{2ik\arg{p(z_{0})}}{\pi},
	\end{gather*}
	for some $k\geq n(a+a^{-1})/2>n$, where $p(z_{0})^{1/\beta}=\pm ia,$ and $a>0.$
\end{lem}

\begin{lem}\label{Nunokawa2}\emph{\cite{Nunokawal2}}
	Let $h(z)$ be analytic in $\mathbb{D}$, with $h(0)=1$ and $h(z)\neq0$. If there exist two points $z_{1},z_{2}\in\mathbb{D}$, such that
	\begin{gather*}
	-\frac{\alpha_{1}\pi}{2}=\arg h(z_{1})<\arg h(z)<\arg h(z_{2})=\frac{\alpha_{2}\pi}{2} \intertext{for $\alpha_{1},\alpha_{2}\in(-1,1)$ and $|z|<|z_{1}|=|z_{2}|$, then we have}
	\frac{z_{1}h'(z_{1})}{h(z_{1})}=-i\frac{\alpha_{1}+\alpha_{2}}{2}k \quad \text{and}\quad 	\frac{z_{2}h'(z_{2})}{h(z_{2})}=i\frac{\alpha_{1}+\alpha_{2}}{2}k,
	\intertext{where}
	k\geq\frac{1-|a|}{1+|a|}\quad
	\text{and}\quad
	a=i\tan{\frac{\pi}{4}\bigg(\frac{\alpha_{2}-\alpha_{1}}{\alpha_{2}+\alpha_{1}}}\bigg).
	\end{gather*}
\end{lem}
The first result of this section is the generalization of 	~\cite[Theorem 1.6]{Nuno}, which produces various corollaries and also generalizes Pommerenke's result ~\cite{pom}: Let $f\in\mathcal{A},$ $g\in\mathcal{C}$ and $0<\alpha\leq1,$ then
$$\left|\arg\left(\frac{f'(z)}{g'(z)}\right)\right|<\frac{\alpha\pi}{2}\qquad\implies\qquad\left|\arg\left(\frac{f(z)}{g(z)}\right)\right|<\frac{\beta(\alpha)\pi}{2}.$$
\begin{theorem}\label{thm1}
	Let $f,g\in\mathcal{A}$ and $0<\alpha\leq1$.  For some $m\in[-1,1)$ and $\beta\in(0,1)$, let $a=i\tan{\frac{m\pi}{4}}$  and $ |g(z)/(zg'(z))|>\beta$.  If \begin{equation}\label{th1}
	\frac{f'(z)}{g'(z)}\prec\bigg(\frac{1+e^{i\mu\pi}z}{1-z}\bigg)^{\alpha(\mu_{1}+\mu_{2})/2},
	\end{equation}
	then
	\begin{equation}\label{f/g}
	\frac{f(z)}{g(z)}\prec\bigg(\frac{1+e^{im\pi}z}{1-z}\bigg)^{\alpha},
	\end{equation}
	where $\mu_{j}=1+(-1)^{j}m+\dfrac{2}{\alpha\pi}\atan{\dfrac{\alpha\beta(1-|a|)\cos\left(\arg{\frac{g(z)}{zg'(z)}}\right)}{1+|a|+(-1)^{j+1}\alpha\beta(1-|a|)\sin\left(\arg{\frac{g(z)}{zg'(z)}}\right)}}\:$ $(j=1,2)$
	and $	\mu=\dfrac{\mu_{2}-\mu_{1}}{\mu_{1}+\mu_{2}}$.
\end{theorem}

\begin{proof}
	Let $p(z):=f(z)/g(z)$. Then in view of Theorem \ref{lemma2}, to  prove (\ref{f/g}), it is sufficient to show $-\alpha(1-m)\pi/2\leq\arg p(z)\leq\alpha(1+m)\pi/2$. On the contrary, if there exists two points $z_{1},z_{2}\in\mathbb{D}$ such that
	\begin{gather*}
	-\alpha(1-m)\frac{\pi}{2}=\arg p(z_{1})<\arg p(z)<\arg p(z_{2})=\alpha(1+m)\frac{\pi}{2}
	\intertext{for $|z|<|z_{1}|=|z_{2}|$
		, then by Lemma \ref{Nunokawa2}, we have}
	\frac{z_{1}p'(z_{1})}{p(z_{1})}=-i\alpha k \quad \text{and} \quad \frac{z_{2}p'(z_{2})}{p(z_{2})}=i\alpha k.
	\end{gather*}
	Since $p(z)=f(z)/g(z)$, thus we have
	\begin{gather*}
	\frac{f'(z)}{g'(z)}=p(z)\bigg(1+\frac{zp'(z)}{p(z)}\frac{g(z)}{zg'(z)}\bigg)\intertext{and}
	\arg\bigg(\dfrac{f'(z)}{g'(z)}\bigg)=\arg p(z)+\arg\bigg(1+\frac{zp'(z)}{p(z)}\frac{g(z)}{zg'(z)}\bigg).
	\end{gather*}
	For $z=z_{1}$, we have	
	\begin{align*}	\arg\bigg(\frac{f'(z_{1})}{g'(z_{1})}\bigg)&\leq-\alpha(1-m)\frac{\pi}{2}+\arg\left(1-i\alpha k\beta\bigg(\cos\left(\arg{\dfrac{g(z_{1})}{z_{1}g'(z_{1})}}\right)+i\sin\bigg(\arg{\dfrac{g(z_{1})}{z_{1}g'(z_{1})}}\bigg)\bigg)\right)\\
	&\leq-\alpha(1-m)\frac{\pi}{2}+\atan\left(\dfrac{\alpha\beta(|a|-1)\cos\bigg(\arg\dfrac{g(z_{1})}{z_{1}g'(z_{1})}\bigg)}{1+|a|+\alpha\beta(1-|a|)\sin\bigg(\arg\dfrac{g(z_{1})}{z_{1}g'(z_{1})}\bigg)}\right),
	\end{align*}
	which contradicts (\ref{th1}). Similarly for $z=z_{2}$, we have
	\begin{align*}	\arg\bigg(\frac{f'(z_{2})}{g'(z_{2})}\bigg)&\geq\alpha(1+m)\frac{\pi}{2}+\arg\bigg(1+i\alpha m\beta\bigg(\cos\bigg(\arg{\dfrac{g(z_{2})}{z_{2}g'(z_{2})}}\bigg)+i\sin\bigg(\arg{\frac{g(z_{2})}{z_{2}g'(z_{2})}}\bigg)\bigg)\bigg)\\
	&\geq\alpha(1+m)\frac{\pi}{2}+\atan\left(\frac{\alpha\beta(1-|a|)\cos\bigg(\arg\dfrac{g(z_{2})}{z_{2}g'(z_{2})}\bigg)}{1+|a|-\alpha\beta(1-|a|)\sin\bigg(\arg\dfrac{g(z_{2})}{z_{2}g'(z_{2})}\bigg)}\right),
	\end{align*}
	which again contradicts (\ref{th1}), that completes the proof.
\end{proof}

\noindent If we choose $g(z)=z$ in Theorem \ref{thm1}, it reduces to the following corollary:
\begin{corollary}\label{cor1}
	Let $f\in\mathcal{A}$ and $0<\alpha\leq1$. For some $m\in[-1,1)$, let $a=i\tan{\frac{m\pi}{4}}$. If \begin{equation*}\label{1st}
	f'(z)\prec\bigg(\frac{1+e^{i\mu\pi}z}{1-z}\bigg)^{\alpha(\mu_{1}+\mu_{2})/2},
	\end{equation*}
	then
	\begin{equation}\label{2nd}
	\frac{f(z)}{z}\prec\bigg(\frac{1+e^{im\pi}z}{1-z}\bigg)^{\alpha},
	\end{equation}	where
	$\mu_{j}=1+(-1)^{j}m+\dfrac{2}{\alpha\pi}\atan{\dfrac{\alpha(1-|a|)}{1+|a|}}\:$ $(j=1,2)$ and $\mu=\dfrac{\mu_{2}-\mu_{1}}{\mu_{1}+\mu_{2}}.$\\
	Further, $f\in\mathcal{S}^{*}(\beta),$ where $\beta=\alpha\pi+\atan\dfrac{\alpha(1-|a|)}{1+|a|}$.
\end{corollary}
\begin{proof}
	The subordination (\ref{2nd}) holds, by using Theorem \ref{thm1}. Now just to prove $f\in\mathcal{S}^{*}(\beta)$. Since
	\begin{gather*}
	\arg\frac{zf'(z)}{f(z)}=\arg\frac{z}{f(z)}+\arg f'(z),
	\intertext{using Theorem \ref{lemma2}, we obtain}
	-\alpha(1+m+\mu_{1})\frac{\pi}{2}<\arg\frac{zf'(z)}{f(z)}<\alpha(1-m+\mu_{2})\frac{\pi}{2},\\
	-\alpha\pi-\atan\frac{\alpha(1-|a|)}{1+|a|}<\arg\frac{zf'(z)}{f(z)}<\alpha\pi+\atan\frac{\alpha(1-|a|)}{1+|a|},
	\end{gather*}
	which implies $f\in\mathcal{S}^{*}(\beta).$
\end{proof}

\begin{remark} If $m=0$, then Corollary \ref{cor1} reduces to ~\cite[Theorem 1.7]{Nuno}.
\end{remark}
\noindent The next corollary is a generalization of \cite[Theorem 1]{Mamoru}.
\begin{corollary}\label{rec}
	Let $f\in\mathcal{A}$, $g\in\mathcal{C}$ and $g\in\mathcal{RS}^{*}(\beta)$, where $0\leq\beta<1$. Suppose $0<\alpha\leq1$ and  for some $m\in[-1,1)$, let $a=i\tan{\frac{m\pi}{4}}$. If \begin{equation*}
	\frac{f'(z)}{g'(z)}\prec\bigg(\frac{1+e^{i\mu\pi}z}{1-z}\bigg)^{(\alpha(\mu_{1}+\mu_{2}))/2},
	\end{equation*}then
	\begin{equation*}
	\frac{f(z)}{g(z)}\prec\bigg(\frac{1+e^{im\pi}z}{1-z}\bigg)^{\alpha}.
	\end{equation*}
	where $\mu_{j}=1-m+\dfrac{2}{\alpha\pi}\atan{\dfrac{\alpha\beta(1-|a|)}{1+|a|+\alpha(1-|a|)}}\:$ $(j=1,2)$ and $\mu=\dfrac{\mu_{2}-\mu_{1}}{\mu_{1}+\mu_{2}}.$
	
\end{corollary}
\begin{proof}
	Since $g\in\mathcal{C}$, from Marx-Strohh\"{a}cker's theorem, we have $\RE(zg'(z)/g(z))>1/2$, which implies that $|g(z)/(zg'(z))-1|<1$. Thus we obtain \begin{equation}\label{recip}
	\bigg|\IM\bigg(\frac{g(z)}{zg'(z)}\bigg)\bigg|<1\quad \text{and}\quad \RE\bigg(\frac{g(z)}{zg'(z)}\bigg)>\beta.\end{equation}
	Now using (\ref{recip}) and the methodology of Theorem \ref{thm1}, the result follows at once.
\end{proof}
\begin{remark} When we take $m=0$ in Corollary \ref{rec}, then the result reduces to \cite[Theorem 1]{Mamoru}.\end{remark}

\begin{theorem}\label{thm2}
	Let $p\in\mathcal{H}_{1}$ and $m\in[-1,1).$ Also, for a fixed $\gamma\in[0,1]$ and $\alpha>0,$ let $\beta>\beta_{0}(\geq0)$, where $\beta_{0}$ is the solution of the equation
	\begin{equation*}
	\alpha\beta(1-m)+\frac{2\gamma}{\pi}\atan{\eta}=0
	\end{equation*} and  for a suitable fixed $\eta\geq0$, let $\lambda(z):\mathbb{D}\to\mathbb{C}$ be a function satisfying
	\begin{equation}\label{eta}
	\frac{\beta\RE\lambda(z)}{1+\beta|\IM\lambda(z)|}\geq\eta.
	\end{equation} If
	\begin{equation}\label{p+lzp'/p}
	(p(z))^{\alpha}\bigg(1+\lambda(z)\frac{zp'(z)}{p(z)}\bigg)^{\gamma}\prec\bigg(\frac{1+e^{i \mu \pi}z}{1-z}\bigg)^\delta,
	\end{equation}
	then
	\begin{equation}\label{p}
	p(z)\prec\bigg(\frac{1+e^{i m \pi}z}{1-z}\bigg)^{\beta}.
	\end{equation}
	where $\mu_{j}=\alpha\beta(1+(-1)^j m)+\dfrac{2\gamma}{\pi}\atan{\eta}\:$ $(j=1,2),$ $\delta=\dfrac{\mu_{1}+\mu_{2}}{2}$ and $\mu=\dfrac{\mu_{2}-\mu_{1}}{\mu_{1}+\mu_{2}}.$
\end{theorem}

\begin{proof}
	According to Theorem \ref{lemma2}, to prove (\ref{p}), it is sufficient to show that $-\beta(1-m)\frac{\pi}{2}<\arg{p(z)}<\beta(1+m)\frac{\pi}{2}.$
	On the contrary if  there exists two points $z_{1}, z_{2}\in\mathbb{D}$ such that
	\begin{gather*}
	-\beta(1-m)\frac{\pi}{2}=\arg{p(z_{1})}<\arg{p(z)}<\arg{p(z_{2})}=\beta(1+m)\frac{\pi}{2}
	\intertext{ for $|z|<|z_{1}|=|z_{2}|,$ then from  Lemma \ref{Nunokawa2}, we have}
	\frac{z_{1}p'(z_{1})}{p(z_{1})}=-ik\beta \quad \text{and} \quad \frac{z_{2}p'(z_{2})}{p(z_{2})}=ik\beta
	\end{gather*}
	where, $k\geq\frac{1-|a|}{1+|a|}$ and $a=i\tan{\frac{m\pi}{4}}.$
	For $z=z_{1}$, using (\ref{eta}) we have  \begin{align*}
	\arg\bigg((p(z_{1}))^{\alpha}\bigg(1+\lambda(z_{1})\frac{z_{1}p'(z_{1})}{p(z_{1})}\bigg)^{\gamma}\bigg)&\leq-\alpha\beta(1-m)\frac{\pi}{2}-\gamma\atan \frac{\beta m\RE\lambda(z_{1})}{1+\beta m|\IM\lambda(z_{1})|}\\
	&\leq-\bigg(\alpha\beta(1-m)\frac{\pi}{2}+\gamma\atan\eta\bigg),
	\end{align*}
	which contradicts (\ref{p+lzp'/p}). Similarly, for $z=z_{2}$, we have
	\begin{align*}
	\arg\bigg({(p(z_{2}))^{\alpha}\bigg(1+\lambda(z_{2})\frac{z_{2}p'(z_{2})}{p(z_{2})}\bigg)^\gamma}\bigg)&\geq\alpha\beta(1+m)\frac{\pi}{2}+k\atan \frac{\beta \gamma\RE\lambda(z_{2})}{1+\beta \gamma|\IM\lambda(z_{2})|}\\
	&\geq\alpha\beta(1+m)\frac{\pi}{2}+\gamma\atan\eta,
	\end{align*}
	which also contradicts (\ref{p+lzp'/p}). Hence the result follows.
\end{proof}
\begin{remark}
	If we take $m=0$ in Theorem \ref{thm2} then the result reduces to  \cite[Theorem 1]{sushma}.
\end{remark}

\section{Radius results}

This section aims to find the largest radius $R$ of a property $\mathfrak{P}$ such that every function of a set $\mathcal{M}$ has the property $\mathfrak{P}$ in the disk $\mathbb{D}_{r},$ where $r\leq R.$ First result of this section generalizes \cite[Theorem 2.1-2.2]{ali}.
\begin{theorem}\label{strong}
	Let $\gamma,\delta\in\mathbb{C}\backslash\{1\}$ and $A,B,C,D\in\mathbb{C}$  with $A\neq B$, $|B|\leq1$, $D\neq C$ and $|D|\leq1.$ If \begin{equation*}
	f(z)\prec (1-\delta)\left(\frac{1+Cz}{1+Dz}\right)^\beta+\delta,
	\end{equation*} then \begin{equation*}
	f(z)\prec (1-\gamma)\left(\frac{1+Az}{1+Bz}\right)^\alpha+\gamma,
	\end{equation*}in $|z|\leq R$, where
	\begin{equation*}
	R=\min\left\{\frac{\alpha|(A-B)(1-\gamma)^{\tfrac{1}{\alpha}}|}{|(C-D)\beta(1-\gamma)^{\tfrac{1}{\alpha}-1}(\delta-1)|+\beta|(1-\gamma)^{\tfrac{1}{\alpha}-1}(AD(\gamma-1)+B(C(1-\delta)+D(\delta-\gamma)))|},1\right\}.
	\end{equation*}
\end{theorem}
\begin{proof}
	Let $P$ and $Q$ be functions defined as\begin{gather*}
	P(z)=(1-\gamma)\left(\frac{1+Az}{1+Bz}\right)^\alpha+\gamma\quad \text{and}\quad
	Q(z)=(1-\delta)\left(\frac{1+Cz}{1+Dz}\right)^\beta+\delta.
	\end{gather*}
	Further, define the function $M$ as \begin{gather*}
	M(z)=P^{-1}(Q(z))=\frac{\left((1-\delta)(1+Cz)^\beta+(\delta-\gamma)(1+Dz)^\beta\right)^{\tfrac{1}{\alpha}}-(1+Dz)^{\tfrac{\beta}{\alpha}}(1-\gamma)^{\tfrac{1}{\alpha}}}{A(1+Dz)^{\tfrac{\beta}{\alpha}}(1-\gamma)^{\tfrac{1}{\alpha}}-B\left((1-\delta)(1+Cz)^\beta+(\delta-\gamma)(1+Dz)^\beta\right)^{\tfrac{1}{\alpha}}}.
	\intertext{Replacing  $z$ by $Rz$ so that $|z|=R\leq1$, we obtain}
	|M(Rz)|\leq\frac{|(C-D)\beta(1-\gamma)^{\tfrac{1}{\alpha}-1}(\delta-1)|R}{\alpha|(A-B)(1-\gamma)^{\tfrac{1}{\alpha}}-\beta|(1-\gamma)^{\tfrac{1}{\alpha}-1}(AD(\gamma-1)+B(C-C\delta+D(\delta-\gamma)))|R}\leq1
	\intertext{for} R\leq\frac{\alpha|(A-B)(1-\gamma)^{\tfrac{1}{\alpha}}|}{|(C-D)\beta(1-\gamma)^{\tfrac{1}{\alpha}-1}(\delta-1)|+\beta|(1-\gamma)^{\tfrac{1}{\alpha}-1}(AD(\gamma-1)+B(C(1-\delta)+D(\delta-\gamma)))|}.
	\end{gather*}
	Thus $f(z)\prec  Q(z)$ for $|z|\leq\min\{R,1\}.$ Hence the result follows.
\end{proof}
\noindent Taking $\gamma=\delta=0$ in the Theorem \ref{strong}, we obtain the following corollary.
\begin{corollary}\label{strc}
	Let $A,B,C,D\in\mathbb{C}$  with $A\neq B$, $|B|\leq1$, $D\neq C$ and $|D|\leq1.$ Then $\mathcal{S}^{*}(C,D,\beta)\subseteq\mathcal{S}^{*}(A,B,\alpha)$ if and only if \begin{equation*}
	\beta(|C-D|+|AD-BC|)\leq\alpha|A-B|.
	\end{equation*}
	In particular, \begin{enumerate}
		\item for $\beta_{1}\in(0,1],$ we have $\mathcal{SS}^{*}(\beta_{1})\subseteq\mathcal{S}^{*}(A,B,\alpha)$ if and only if \begin{equation*}
		\beta(2+|A+B|)\leq\alpha|A-B|.
		\end{equation*}
		\item for $\alpha_{1}\in(0,1],$ we have $\mathcal{S}^{*}(C,D,\beta)\subseteq\mathcal{SS}^{*}(\alpha_{1})$ if and only if \begin{equation*}
		\beta(|C+D|+|C-D|)\leq2\alpha.
		\end{equation*}
	\end{enumerate}
\end{corollary}
\begin{remark}
	Let $\alpha_{1},\alpha_{2}\in[0,1)$, $A=1-2\alpha_{1},$ $C=1-2\alpha_{2},$  $B=D=-1$ and $\alpha=\beta=1$, then the corollary \ref{strc} reduces to the well known fact that $\mathcal{S}^{*}(\alpha_{2})\subseteq\mathcal{S}^{*}(\alpha_{1})$ if and only if $\alpha_{2}\leq\alpha_{1}.$
\end{remark}

\begin{corollary}
	Let  $A,B\in\mathbb{C}$  with $A\neq B$, $|B|\leq1$ and $\beta_{2}>0$. If $f\in\mathcal{S}^{*}(A,B,\alpha),$ then \begin{enumerate} \item $f\in\mathbb{M}(\beta_{2})$ in $|z|\leq R_{\mathbb{M}}(\beta_{2})$  for $\beta>1,$ where
		\begin{equation}
		R_{\mathbb{M}}(\beta_{2})=\min\left\{\frac{a|A-B|}{2(\beta_{2}+1)+|A-B(2\beta-1)|},1\right\}.
		\end{equation}
		\item$f\in\mathcal{RS}^{*}(\beta_{2})$ in $|z|\leq R_{\mathcal{RS}^{*}}(\beta_{2})$  for $0\leq\beta<1,$ where
		\begin{equation}
		R_{\mathcal{RS}^{*}}(\beta_{2})=\min\left\{\frac{a|A-B|}{2\beta_{2}+|A-B(2\beta-1)|},1\right\}.
		\end{equation}
	\end{enumerate}
\end{corollary}

\begin{theorem}
	Let $f(z)\in\mathcal{A}$ with $f'(z)\neq0$ in $\mathbb{D}.$ Also let $0<A\leq1,$ $-1\leq B<0,$  $\alpha=0.38344486\cdots$ and $\beta\geq0.61655\cdots$ which satisfy the conditions $\atan{\alpha}=(1-2\alpha)/2$ and $\atan\alpha\leq(\beta-\alpha)\pi/2$, respectively. If
	\begin{equation}\label{f'}
	f'(z)\prec\bigg(\frac{1+Az}{1+Bz}\bigg)^\beta,
	\end{equation}
	then $f(z)$ is starlike in $|z|<r_{0}$, where \begin{equation}\label{r0}
	r_{0}=\frac{-(A-B)+\sqrt{(A-B)^2+4AB\left(\sin(\alpha+\dfrac{2}{\pi}\atan\alpha)\dfrac{\pi}{2\beta}\right)^2}}{2AB\sin\left((\alpha+\dfrac{2}{\pi}\atan\alpha)\dfrac{\pi}{2\beta}\right)},
	\end{equation}
	is the smallest positive root of the equation \begin{equation}\label{root}
	\bigg(\alpha+\frac{2}{\pi}\atan\alpha\bigg)\frac{\pi}{2}=\beta\asin\bigg(\frac{(A-B)x}{1-AB x^2}\bigg).
	\end{equation}
\end{theorem}

\begin{proof}
	By Theorem \ref{argjan}, the subordination (\ref{f'}) yields that
	\begin{equation}\label{asin}
	|\arg{f'(z)}|\leq\beta\asin\frac{(A-B)|z|}{1-AB|z|^2}.
	\end{equation}
	Let $p(z)=f(z)/z$, we suppose that there exists a point $z_{0}$, $|z_{0}|=r_{0}<1,$ such that $|\arg{p(z)}|<\alpha\pi/2$ for $|z|<r_{0}$ and $|\arg{p(z_{0})}|=\alpha\pi/2$ then by Lemma \ref{Nunokawa1}, we have
	\begin{equation}\label{zp'/p}
	\frac{z_{0}p'(z_{0})}{p(z_{0})}=\frac{2i k\arg{p(z_{0})}}{\pi},
	\end{equation}
	where $ k\geq m(a+a^{-1})/2\geq1$, $p(z_{0})=\pm ia$ and $a>0.$ Also if $\arg{p(z_{0})}=\alpha\pi/2$ then using (\ref{root}), we have
	\begin{align*}
	\arg f'(z_{0})&=\arg(p(z_{0})+zp'(z_{0}))\\
	&=\arg p(z_{0})+\arg\bigg(1+\frac{z_{0}p'(z_{0})}{p(z_{0})}\bigg)\\
	&=\frac{\alpha\pi}{2}+\arg(1+i\alpha k)\\
	&\geq\frac{\alpha\pi}{2}+\atan{\alpha}\\
	&=\beta\asin\bigg(\frac{(A-B)r_{0}}{1-AB r_{0}^2}\bigg),
	\end{align*}
	which contradicts (\ref{asin}). Similar contrary conclusion will arrive for the case when $\arg{p(z_{0})}=-\alpha\pi/2$. Therefore we have \begin{equation}\label{f/z}
	|\arg p(z_{0})|=\bigg|\arg \bigg(\frac{f(z)}{z}\bigg)\bigg|<\frac{\alpha\pi}{2}.
	\end{equation}
	Now using (\ref{root}), (\ref{asin}) and (\ref{f/z}), we have
	\begin{align*}
	\bigg|\arg\bigg(\frac{zf'(z)}{f(z)}\bigg)\bigg|&\leq|\arg f'(z)|+\bigg|\arg\bigg(\frac{z}{f(z)}\bigg)\bigg|\\
	&<\bigg(2\alpha+\frac{2}{\pi}\atan{\alpha}\bigg)\frac{\pi}{2}.
	\end{align*}
	For $f$ to be starlike, we must have $2\alpha+(2\atan{\alpha})/\pi=1$, which gives $\alpha=0.38344486\cdots$ and since
	\begin{equation*}
	\bigg(\alpha+\frac{2}{\pi}\atan\alpha\bigg)\frac{\pi}{2}=\beta\asin\bigg(\frac{(A-B)r_{0}}{1-AB r_{0}^2}\bigg),
	\end{equation*}
	we have $f(z)$ is starlike in $|z|<r_{0},$ where $r_{0}$ is given in (\ref{r0}), that completes the proof.
\end{proof}

\section{Subordination results}

In this section we discuss  subordination related results for the Janowski function. The book \cite{miller} provides  numerous results pertaining to differential subordination and some of the results which we need in context of our study are listed below:
\begin{lem}\label{suff}  \emph{\cite[Theorem 3.1d,p.76]{miller}} Let $h$ be analytic and starlike univalent in $\mathbb{D}$ with $h(0) = 0$. If
	$g$ is analytic in $\mathbb{D}$ and $zg'(z) \prec h(z)$, then
	\begin{equation*}
	g(z) \prec g(0) +\int_{0}^{z}\frac{h(t)}{t}dt.
	\end{equation*}
\end{lem}
\begin{lem}\label{MM}
	\emph{\cite[Theorem 3.4h,p.132]{miller}} Let $g(z)$ be univalent in $\mathbb{D}$,  $\Phi$ and $\Theta$ be  analytic in a domain $\Omega$ containing $g(\mathbb{D})$ such that $\Phi(w)\neq0,$ when $w\in g(\mathbb{D})$. Now letting $G(z)=zg'(z)\cdot\Phi(g(z))$, $h(z)=\Theta(g(z))+G(z)$ and either  $h$ or $g$ is convex. Further, if \begin{equation*}
	\RE\frac{zh'(z)}{G(z)}=\RE\bigg(\frac{\Theta'(g(z))}{\Phi(g(z))}+\frac{zG'(z)}{G(z)}\bigg)>0.
	\end{equation*}  as well as  $p$ is analytic  in $\mathbb{D},$ with $p(0)=g(0),$ $p(\mathbb{D})\subset \Omega$ and
	\begin{equation*}
	\Theta(p(z))+zp'(z)\cdot\Phi( p(z))\prec \Theta(g(z))+zg'(z)\cdot\Phi(g(z)):= h(z)
	\end{equation*}
	then $p\prec g$, and $g$ is the best dominant.
\end{lem}
\noindent Antonino and Romaguera in \cite{antonino} gave the following result:
\begin{lem}\label{anto}\emph{\cite{antonino}}
	Let $p(z)$ be analytic in $\mathbb{D}, p(0)=c,$ let $q(z)$ be univalent and convex in $\mathbb{D}$ and let $g(z)$ be analytic in $\mathbb{D}$ such that $zg'(z)/g(z)$ is analytic and different from zero in $\mathbb{D}$. Let $\beta$ and $\gamma$ be complex constants such that \begin{equation*}
	\RE\bigg(\frac{\xi g'(\xi)}{g(\xi)}(\beta q(z)+\gamma)\bigg)>0 \quad (z, \xi\in\mathbb{D}).
	\end{equation*}
	then\begin{equation*}
	p(z)+\frac{g(z)}{zg'(z)}\bigg(\frac{zp'(z)}{\beta p(z)+\gamma}\bigg)\prec q(z)\implies p(z)\prec q(z).
	\end{equation*}
\end{lem}

\noindent 	Several authors studied various combinations of $zf'(z)/f(z)$ and $1+zf''(z)/f'(z)$. One such combination is their quotient, which was first investigated by Silverman in ~\cite{silver}. Silverman considered  the following class for $\beta\in(0,1]$,
\begin{equation*}
\mathcal{G}_{\beta}=\bigg\{f\in\mathcal{A}:\frac{1+zf''(z)/f'(z)}{zf'(z)/f(z)}-1\prec\beta z\bigg\}.
\end{equation*}
\begin{theorem}\label{sok}
	Let $|A|\leq1$ and $0\leq b\leq1$ with $A+b\neq0$ and $\beta(1+b)^{\alpha-1}(1+|A|)^{\alpha+1}\leq\alpha|A+b|.$ If $f\in\mathcal{G}_{\beta},$ then $f\in\widetilde{\mathcal{SS}}^{*}(A,b,\alpha).$
\end{theorem}

\begin{proof}
	For $f\in\mathcal{G}_{\beta}$, we define the function $p(z)= zf'(z)/f(z)$. By standard calculations we obtain that
	\begin{equation*}
	\frac{1+zf''(z)/f'(z)}{zf'(z)/f(z)}-1=\frac{zp'(z)}{p^2(z)}.
	\end{equation*}
	For $f\in\widetilde{\mathcal{SS}}^{*}(A,b,\alpha)$, it is enough to show $p(z)\prec ((1+Az)/(1-bz))^{\alpha}:=q(z)$. For if, there exist points $z_{0}\in\mathbb{D}$, $\varsigma_{0}\in\partial\mathbb{D}\backslash\{-1/A,1/b\}$ and $m\geq1$ such that $p(|z|<|z_{0}|)\subset q(\mathbb{D})$ with $p(z_{0})=q(\varsigma_{0})$ and $z_{0}p'(z_{0})=m\varsigma_{0}q'(\varsigma_{0}).$ If
	\begin{gather*}
	\dfrac{\alpha|A+b|}{(1+|A|)^{\alpha+1}(1+b)^{\alpha-1}}\geq\beta
	\implies
	\bigg|\frac{\alpha\varsigma_{0}(A+b)}{(1+A\varsigma_{0})^{\alpha+1}(1-b\varsigma_{0})^{\alpha-1}}\bigg|\geq\beta.
	\intertext{Thus,}
	\bigg|\frac{m\varsigma_{0}q'(\varsigma_{0})}{q^2(\varsigma_{0})}\bigg|\geq\beta\quad \text{for all $m\geq1$}.
	\intertext{or, equivalently,}
	\bigg|\frac{z_{0}p'(z_{0})}{p^2(z_{0})}\bigg|\geq\beta,
	\end{gather*}
	which contradicts $f\in\mathcal{G}_{\beta}$. Thus $p(z)\prec q(z)$ and that completes proof.
\end{proof}

\begin{remark}
	If $\alpha=1$ and $-1\leq B<A\leq1$ in Theorem \ref{sok} then the result reduces to \cite[Theorem 3.2,p.9]{sokol}.
\end{remark}
\begin{theorem}
	If $f\in\mathcal{G}_{\beta}$, then $f$ is starlike of reciprocal order $\alpha.$
\end{theorem}

\begin{proof}
	For $f\in\mathcal{G}_{\beta}$, we define the function $p(z)$ by $f(z)/(zf'(z))=\alpha+(1-\alpha)p(z).$ Clearly $p\in\mathcal{H}_{1}$  and $\alpha+(1-\alpha)p(z)\neq0$ for $z\in\mathbb{D}.$ Now by a simple computation we obtain that
	\begin{equation*}
	1-\frac{1+zf''(z)/f'(z)}{zf'(z)/f(z)}=(1-\alpha)zp'(z).
	\end{equation*}
	Since $f\in\mathcal{G}_{\beta}$, therefore we have
	\begin{equation}\label{(1-a)zp'}
	(1-\alpha)zp'(z)\prec \beta z.
	\end{equation}
	Now using Lemma \ref{suff},
	(\ref{(1-a)zp'}) leads to \begin{equation*}
	p(z)\prec1+\frac{\beta}{1-\alpha}z,
	\end{equation*}
	which shows that
	\begin{equation*}
	|p(z)-1|<\frac{\beta}{1-\alpha}.
	\end{equation*}
	Now using Theorem \ref{argjan}, we obtain that
	\begin{equation*}
	\bigg|\Arg\bigg(\frac{f(z)}{zf'(z)}-\alpha\bigg)\bigg|=|\Arg p(z)|<\asin{\frac{\beta}{1-\alpha}}=\frac{\delta \pi}{2}.
	\end{equation*}
	Since $\alpha\in[0,1), $ and $\beta\in(0,1]$. Therefore $\delta\in(0,1]$, which completes our result.
\end{proof}

In ~\cite{vikram}, authors have considered Janowski function with complex parameters and therefore the corresponding Janowski disk is non symmetric with respect to real axis. However their findings were based on the Janowski disk that are symmetric with respect to real axis, which is possible only if the parameters are real in Janowski function. Therefore by eliminating this limitation, we obtain the following result as an extension of ~\cite[Lemma 2.17]{vikram}.
\begin{theorem}Let $\alpha\in(0,1],$ $l,m\in[-1,1]$ and  $a,b,c,d\in[0,1]$  such that $a e^{il\pi}+b\neq0$ and $c e^{i m\pi}+d\neq0$.   Let $Q\in\mathcal{H}[1,n]$ satisfy
	\begin{equation}\label{Q}
	Q(z)\prec\frac{1+a e^{i l\pi}z}{1-bz}
	\end{equation}  and
	\begin{equation}\label{Qpa}
	Q(z)p^{\alpha}(z)\prec\frac{1+c e^{i m\pi}z}{1-dz},
	\end{equation}  for $p\in\mathcal{H}_{1}$. 	If
	\begin{equation}\label{musum}
	\mu:=\asin\sqrt{\frac{c^2+d^2+2cd\cos(m\pi)}{1+c^2d^2+2cd\cos(m\pi)}}+\asin\sqrt{\frac{a^2+b^2+2ab\cos(l\pi)}{1+a^2b^2+2ab\cos(l\pi)}}\leq\frac{\alpha\pi}{2},
	\end{equation}then  $$\RE(e^{-i\gamma\pi/2}p(z))>0,$$ where
	$ \gamma=(\atan A-\atan B)/\mu$ with $A=\dfrac{cd\sin(m\pi)}{cd\cos(m\pi)+1}$ and $B=\dfrac{ab\sin(l\pi)}{ab\cos(l\pi)+1}.$
\end{theorem}
\begin{proof}
	From Theorem \ref{argjan}, (\ref{Q}) yields \begin{equation}\label{argQ}
	\bigg|\arg Q(z)-\atan\left(\frac{ab\sin(l\pi)}{ab\cos(l\pi)+1}\right)\bigg|<\asin\sqrt{\frac{a^2+b^2+2ab\cos(l\pi)}{1+a^2b^2+2ab\cos(l\pi)}}.
	\end{equation}
	Similarly, from (\ref{Qpa}), we obtain \begin{equation}\label{argQpa}
	\bigg|\arg Q(z)+\alpha\arg p(z)-\atan\left(\frac{cd\sin(m\pi)}{cd\cos(m\pi)+1}\right)\bigg|<\asin\sqrt{\frac{c^2+d^2+2cd\cos(m\pi)}{1+c^2d^2+2cd\cos(m\pi)}}.
	\end{equation}
	After some computations using (\ref{argQ}) and (\ref{argQpa}), we obtain
	
	\begin{equation}
	-\frac{\pi}{2}\left(\frac{2}{\alpha\pi}\left(\mu-\gamma\mu\right)\right)\leq\arg p(z)\leq \frac{\pi}{2}\left(\frac{2}{\alpha\pi}\left(\mu+\gamma\mu\right)\right)
	\end{equation}  which eventually yields
	\begin{equation*}
	p(z)\prec\frac{1+e^{i\gamma\pi }z}{1-z},
	\end{equation*}that completes the proof.
\end{proof}

\begin{theorem}\label{owathm}
	Let $p(z)\in\mathcal{H}_{1}$, $A\in\mathbb{C}$ and   $0\leq b\leq 1$  with $|A|\leq1$, $A+b\neq0$  and $\RE(1+Ab)\geq|A+b|$. Further let $\alpha,\gamma$ are two real parameters lying in $[0,1]$ and $\mu,\delta,\rho$ and $\eta$ are complex parameters such that $\RE(\mu/\eta)>0$, $\RE\delta>0$ and $\RE\rho\geq0$. If \begin{align*}
	\mu (p(z))^{\alpha}(\delta+\rho p(z))+\eta zp'(z)(p(z))^{\alpha-1}\prec h(z),
	\end{align*}  then
	$p\in\widetilde{\mathcal{SS}}^{*}(A,b,\gamma),$
	where \begin{equation*}
	h(z)=\bigg(\frac{1+Az}{1-b z}\bigg)^{\alpha\gamma}\bigg(\mu\delta+\mu\rho\bigg(\frac{1+Az}{1-b z}\bigg)^{\gamma}+\eta\gamma\frac{(A+b)z}{(1+Az)(1-b z)}\bigg).
	\end{equation*}
\end{theorem}
\begin{proof}
	Let us choose
	$g(z)=((1+Az)/(1-b z))^{\gamma}$, $\Phi(w)=\eta w^{\alpha-1}$ and $\Theta(w)=\mu w^{\alpha}(\delta+\rho w)$ then clearly $g\in\mathcal{P},$ univalent and convex in $\mathbb{D}$.
	$\Phi,\Theta$ are analytic in a domain $\Omega$ containing $g(\mathbb{D})$, with $\Phi(w)\neq0$ when $w\in g(\mathbb{D})$.  If $G(z)=zg'(z)\Phi(g(z))$ then
	$$\RE\frac{zG'(z)}{G(z)}=\RE\bigg(\frac{\alpha+1}{1-bz}-\frac{\alpha-1}{1+Az}-1\bigg)\geq\frac{\alpha+1}{1+b}-\frac{\alpha-1}{1+|A|}-1\geq0$$
	and $$\RE\frac{zh'(z)}{G(z)}=\RE\bigg(\frac{\mu\alpha\delta}{\eta}+\frac{\mu(\alpha+1)\rho g(z)}{\eta}+\frac{zG'(z)}{G(z)}\bigg)>0.$$
	Thus by Lemma \ref{MM}, we obtain that $p\prec g$ and $g$ is the best dominant.
\end{proof}
\noindent By taking $\mu=1,\delta=1-\lambda,\rho=\lambda,\eta=\lambda$ and $p(z)=zf'(z)/f(z)$ in Theorem \ref{owathm}, where $\lambda\in[0,1],$ we obtain the following more modified and simplified form of results in \cite{kuroki}.
\begin{corollary}
	Let $f(z)\in\mathcal{A}$, such that $f(z)/z\neq0$ in $\mathbb{D}$, $A\in\mathbb{C}$ and   $0\leq b\leq 1$  with $|A|\leq1$, $A+b\neq0$  and $\RE(1+Ab)\geq|A+b|$. Suppose also that the real parameters $\lambda,\alpha$ and $\gamma$ are such that they  lies in $[0,1]$. If \begin{align*}
	\bigg(\frac{zf'(z)}{f(z)}\bigg)^{\alpha}\bigg(1+\lambda \frac{zf''(z)}{f'(z)}\bigg)\prec h(z),
	\end{align*}  then $f\in\widetilde{\mathcal{SS}}^{*}(A,b,\gamma),$
	where \begin{equation*}
	h(z)=\bigg(\frac{1+Az}{1-b z}\bigg)^{\alpha\gamma-1}\bigg((1-\lambda)\frac{1+Az}{1-b z}+\frac{\lambda(1+Az)^{1+\gamma}(1-bz)^{1-\gamma}+\lambda\gamma(A+b)z}{(1-b z)^2}\bigg).
	\end{equation*}
\end{corollary}
\noindent By taking $\delta=1$ and $\rho=0$ in Theorem \ref{owathm}, we obtain the following result.
\begin{corollary}\label{owathm1}
	Let $p(z)\in\mathcal{H}_{1}$ and $0\leq b\leq 1$ with $b+e^{i m\pi}\neq0,$ where $-1\leq m\leq1$. Also let $\alpha>-1$ and if \begin{equation*}
	\mu (p(z))^{\alpha}+\eta zp'(z)(p(z))^{\alpha-1}\prec \bigg(\frac{1+e^{i m\pi}z}{1-b z}\bigg)^{\alpha\gamma}\bigg(\mu+\eta\gamma\frac{(b+e^{i m\pi})z}{(1+e^{i m\pi}z)(1-b z)}\bigg):=h(z),
	\end{equation*} for some $\mu,\eta\in\mathbb{C}$ such that $\RE(\mu/\eta)\geq0$, then
	\begin{equation*}
	\RE e^{-i\beta}(p(z))^{1/\gamma}>0,
	\end{equation*} where $\beta=\atan\dfrac{b\sin{(m\pi)}}{b\cos{(m\pi)}+1}$.
	And the inequality is sharp for the function $p(z)$ defined by \begin{equation*}
	p(z)=\bigg(\frac{1+e^{im\pi}z}{1-b z}\bigg)^{\gamma}.\end{equation*}
\end{corollary}
\begin{proof}
	By Theorem \ref{owathm}	we have $p(z)\prec ((1+e^{im\pi}z)/(1-b z))^{\gamma}:=g(z)$ and $g$ is the best dominant. Further by Theorem \ref{lemma3} we obtain the desired conclusion.
\end{proof}

\noindent By taking $\mu=1-\lambda,\alpha=1$ and $\eta=\lambda$ in Corollary \ref{owathm1},  we obtain the following result.
\begin{corollary}\label{cor2}
	Let $p(z)\in\mathcal{H}_{1}$ and $0\leq b\leq 1$ with $b+e^{i m\pi}\neq0,$ where $-1\leq m\leq1$. Now if \begin{equation*}
	(1-\lambda)p(z)+\lambda zp'(z)\prec \bigg(\frac{1+e^{i m\pi}z}{1-b z}\bigg)^{\gamma}\bigg(1-\lambda+\lambda\gamma\frac{(b+e^{i m\pi})z}{(1+e^{i m\pi}z)(1-b z)}\bigg):=h(z),
	\end{equation*} for some $0<\lambda\leq1$ and $0<\gamma\leq1$, then
	\begin{equation*}
	\RE e^{-i\beta}(p(z))^{1/\gamma}>0,
	\end{equation*} where $\beta=\atan\dfrac{b\sin{(m\pi)}}{b\cos{(m\pi)}+1}$.
	Further the inequality is sharp for the function $p(z)$ given by \begin{equation*}
	p(z)=\bigg(\frac{1+e^{im\pi}z}{1-b z}\bigg)^{\gamma}.\end{equation*}
\end{corollary}

\begin{remark} When $m=0$ and $b=1$ then Corollary \ref{cor2} reduces to the ~\cite[Theorem 1]{owa}.\end{remark} In case when $m=0$, $\gamma=1$ and $\lambda=0.5$, Corollary \ref{cor2} yields:
\begin{corollary}
	Let $p(z)\in\mathcal{H}_{1}$ and $0\leq b\leq1$. Suppose \begin{equation*}
	p(z)+zp'(z)\prec \frac{1+2z-bz^2}{(1-bz)^2},
	\end{equation*}then $$p(z)\prec\frac{1+z}{1-bz}.$$
\end{corollary}

\begin{theorem}\label{bb}
	Let $\alpha\in(0,1],\beta,\gamma\in\mathbb{C}$ with $\beta\neq0$, $\RE(\beta+\gamma)>0$ and $A,B\in\mathbb{C}$ with $A\neq B$, $|B|\leq1$ and $|A-B|\leq1-\RE(A\overline{B})$. Let $\lambda(z)$ be Carath\'{e}odory and if $p(z)$ is analytic such that
	\begin{equation*}
	p(z)+\lambda(z)\frac{ zp'(z)}{\beta p(z)+\gamma}\prec \bigg(\dfrac{1+Az}{1+Bz}\bigg)^{\alpha},	\end{equation*}then $$p(z)\prec q(z):=H(z)\bigg(\int_{0}^{z}\frac{H(t)g'(t)dt}{g(t)}\bigg)^{-1}\prec\bigg(\dfrac{1+Az}{1+Bz}\bigg)^{\alpha},$$
	where \begin{gather*}
	g(z)=z \exp\int_{0}^{z}\frac{(\lambda(t))^{-1}-1}{t}dt\quad \text{and}\quad
	H(z)=g(z)\exp{\int_{0}^{z}\frac{g'(t)\left(\left(\frac{1+At}{1+Bt}\right)^{\alpha}-1\right)dt}{g(t)}}.
	\end{gather*}
	The function $q(z)$ is convex and is the best (1,n)-dominant.
\end{theorem}

\begin{proof}
	Let  $f\in\mathcal{A}_{n}$ with $f(z)/z\neq0$ in $\mathbb{D}$ and \begin{equation}\label{1}
	p(z)=\frac{F'(z)g(z)}{F(z)g'(z)},
	\end{equation}where \begin{equation}\label{2}
	F(z)=\bigg(\frac{\beta+\gamma}{g^{\gamma}(z)}\int_{0}^{z}f^{\beta}(t)g^{\gamma-1}(t)g'(t)dt\bigg)^{\frac{1}{\beta}}.
	\end{equation}
	Since $\RE(\beta+\gamma)>0$, by using an argument similar to that of \cite[Lemma 1.2c,p.11]{miller}, we can show that the function $S(z)$ defined by \begin{equation}\label{3}
	S(z)=\frac{z^{\beta+\gamma-1}}{z^{\beta+\gamma}g^{\gamma-1}g'(z)f^{\beta}(z)}\int_{0}^{z}\bigg(\frac{f(t)}{t}\bigg)^{\beta}t^{\beta+\gamma-1}\bigg(\frac{g(t)}{t}\bigg)^{\gamma-1}g'(t)dt,
	\end{equation}
	is analytic in $\mathbb{D}$ and $S\in\mathcal{H}[\frac{1}{\beta+\gamma},n].$ From (\ref{2}) and (\ref{3}) we obtain\begin{equation*}
	\frac{F(z)}{z}=\bigg((\beta+\gamma)S(z)\frac{zg'(z)}{g(z)}\bigg)^{\frac{1}{\beta}}\bigg(\frac{f(z)}{z}\bigg).
	\end{equation*}
	Since both the expressions on the right are analytic and nonzero, therefore we conclude that $F\in\mathcal{A}_{n}$ and $F(z)/z\neq0$ for $z\in\mathbb{D}.$
	Now by a simple computation using \ref{1} and \ref{2}, we obtain \begin{equation*}
	p(z)+\lambda(z)\frac{zp'(z)}{\beta p(z)+\gamma}=\frac{f'(z)g(z)}{f(z)g'(z)}\prec\bigg(\dfrac{1+Az}{1+Bz}\bigg)^{\alpha}.
	\end{equation*}
	Since $zg'(z)/g(z)$ and $((1+Az)/(1+Bz))^{\alpha}$ are Carath\'{e}odory, also  $\RE(\beta+\gamma)>0$, thus we have $$\RE\bigg(\dfrac{zg'(z)}{g(z)}\bigg(\beta\bigg(\dfrac{1+Az}{1+Bz}\bigg)^{\alpha}+\gamma\bigg)\bigg)>0.$$ Therefore the result holds by Lemma \ref{anto}.
\end{proof}
\noindent Theorem \ref{bb} is a generalisation of \cite[Theorem 2]{kuroki2} and also we have the following result.
\begin{corollary}
	Let $\alpha\in(0,1],\beta,\gamma\in\mathbb{C}$ with $\RE(\beta+\gamma)>0$ and $A,B\in\mathbb{C}$ with $A\neq B, |B|\leq1$ and $|A-B|\leq1-\RE(A\overline{B})$. Suppose that $g(z)\in\mathcal{S}^{*}$ and $f\in\mathcal{A}$ satisfies
	\begin{equation*}
	\frac{f'(z)g(z)}{f(z)g'(z)}\prec\bigg(\frac{1+Az}{1+Bz}\bigg)^{\alpha}
	\end{equation*}
	then the function $F(z)$ defined in (\ref{2}) is in $\mathcal{A}$, $F(z)/z\neq0,$ for $z\in\mathbb{D}$ and \begin{equation*}
	\frac{F'(z)g(z)}{F(z)g'(z)}\prec\bigg(\frac{1+Az}{1+Bz}\bigg)^{\alpha}.
	\end{equation*}
\end{corollary}
\noindent As stated in (\ref{takeb}), we can consider $B\in[-1,0]$ and by taking $\lambda(z)\equiv1$, $\beta=1$ and  $\gamma=0$ in Theorem \ref{bb}, and letting $p(z)=zf'(z)/f(z)$, for $f\in\mathcal{A}$, we obtain the following subordination implication:

\begin{corollary}\label{3.17}
	If $\alpha\in(0,1]$ and $A\in\mathbb{C}$, $b\in[0,1]$ such that $A+b\neq0$ and $|A+b|\leq1+\RE (Ab)$. Also, if $f(z)\in\mathcal{A}$ satisfies
	\begin{equation*}
	1+\frac{zf''(z)}{f'(z)}\prec\bigg(\frac{1+Az}{1-bz}\bigg)^{\alpha}
	\end{equation*}
	then
	\begin{equation*}
	\frac{zf'(z)}{f(z)}\prec\frac{zK'(z)}{K(z)},
	\end{equation*}where \begin{equation*}
	K(z)=\begin{cases*}\int_{0}^{z}\exp\bigg(\int_{0}^{w}\dfrac{(\frac{1+A t}{1-b t})^{\alpha}-1}{t}dt\bigg)dw,& $A\neq0$, $b\neq0$\\
	\int_{0}^{z}\exp(\alpha b  t [_{3}F_{2}[{1, 1, 1 + \alpha}, {2, 2}, b t]])dt,& $A=0$, $b\neq0$,\\
	\int_{0}^{z}\exp(\alpha A  t [_{3}F_{2}[{1, 1, 1 - \alpha}, {2, 2}, -A t]])dt,& $A\neq0$, $b=0$.
	\end{cases*}
	\end{equation*}
\end{corollary}
\begin{remark} When $\alpha=1$, Corollary \ref{3.17} reduces to  \cite[Corollary 3.2]{kuroki2}, for which
	\begin{equation*}
	K(z)=\begin{cases*}((1-bz)^{-A/b}-1)/A,& $A\neq0$, $b\neq0$,\\
	-\log{(1-bz)}/b,& $A=0$, $b\neq0$,\\
	(e^{Az}-1)/A,& $A\neq0$, $b=0$.
	\end{cases*}
	\end{equation*}\end{remark}
\noindent Moreover, for $A=1-2\beta$ $(0\leq\beta<1)$ and $b=1$, with $\alpha=1$ in Corollary \ref{3.17}, we obtain that, if $f(z)\in\mathcal{C}(\beta)$ then $f(z)\in\mathcal{S}^{*}(\gamma)$, where
\begin{equation*}
\gamma=\gamma(\beta)=\begin{cases*}
\frac{1-2\beta}{2(2^{1-2\beta}-1)},& $\beta\neq\frac{1}{2}$,\\
1/(2\log{2}),& $\beta=\frac{1}{2}$.
\end{cases*}
\end{equation*}
Note that this relationship was proved by MacGregor in \cite{Mac}.

\section*{Acknowledgement}
The work of the second author is supported by The Council of Scientific and Industrial Research(CSIR). Ref.No.: 08/133(0030)/2019-EMR-I.

\bibliographystyle{model1-num-names}

\end{document}